\documentclass[12pt]{amsproc}
\usepackage{geometry} 
\usepackage{amsmath,amssymb,amsthm}
\usepackage[english]{babel}
\usepackage{tikz}
\usepackage{makecell}
\usepackage{hyperref}
\hypersetup{hypertex=true,colorlinks=true,linkcolor=blue,anchorcolor=blue,citecolor=blue}
\usepackage{graphicx}
\usepackage{setspace}
\usepackage{array}
\usepackage{multirow}
\usepackage{mathrsfs}

\title[pointwise convergence of Schrödinger mean with complex time]{Sharp pointwise convergence of Schrödinger mean with complex time in higher dimensions}
\author{Meng Wang, Zhichao Wang}

\newtheorem{theorem}{Theorem}[section]
\newtheorem{lemma}[theorem]{Lemma}
\newtheorem{proposition}[theorem]{Proposition}
\newtheorem*{remark}{Remark}

\begin{document}
	\begin{sloppypar}
		\maketitle
		\let\thefootnote\relax
		\begin{abstract}
			
			In this paper, we establish the almost everywhere convergence of solutions to the Schrödinger operator with complex time $ P_{\gamma}f(x,t) $ in higher dimensions, under the assumption that the initial data $f$ belongs to the Sobolev space $ H^{s}(\mathbb{R}^d) $.

		\end{abstract}
		\section{Introduction}
		\numberwithin{equation}{section}
			
		The solution to the Schr\"odinger equation  
\[
\left\{\begin{array}{ll}
i u_{t}+(-\Delta) u=0, & (x, t) \in \mathbb{R}^{d} \times \mathbb{R}^{+}, \\[2mm]
u(x, 0)=f(x), & x \in \mathbb{R}^{d}
\end{array}\right.
\]  
can be expressed formally as  
\begin{equation}\label{fracsop}
e^{i t(-\Delta)} f(x)=\frac{1}{(2 \pi)^{d}} \int_{\mathbb{R}^{d}} e^{i\left(x \cdot \xi+t|\xi|^{2}\right)} \hat{f}(\xi) \, d\xi .
\end{equation}

Carleson \cite{MR0576038} first posed the problem of determining the optimal regularity exponent \( s \) such that  
\[
\lim_{t \to 0} e^{i t (-\Delta)} f(x) = f(x) \quad \text{almost everywhere for all } f \in H^{s}(\mathbb{R}).
\]  
He established convergence for \( s \geq \frac{1}{4} \), and later Dahlberg and Kenig \cite{MR0654188} proved that this threshold is sharp.

For dimensions \( d > 1 \), the question of almost everywhere convergence becomes considerably more difficult. Considerable work has been devoted to this problem by numerous authors (see, e.g., \cite{MR1315543,MR3241836,MR3702674,MR2264734,MR3613507,MR3903115,MR0904948,MR0934859}). In particular, due to counterexamples constructed by Bourgain \cite{MR3574661}, Du and Zhang \cite{MR3961084} ultimately proved that \( s > \frac{d}{2(d+1)} \) is the critical regularity in higher dimensions. The endpoint case \( s = \frac{d}{2(d+1)} \), however, remains unresolved.

A natural extension of the Schr\"odinger operator involves allowing the time variable to take complex values with positive imaginary part. For instance, replacing \( t \) with \( it \) in (\ref{fracsop}) yields the solution of a linear fractional dissipative equation. For this case, Miao, Yuan, and Zhang \cite{MR2372358} observed that \( f \in L^{2} \) already guarantees pointwise convergence.  
If instead we substitute \( t \) with \( e^{i\theta} t \) in (\ref{fracsop}), the corresponding solution is related to the linear complex Ginzburg--Landau equation; further details can be found in \cite{MR3032976}.  

Replacing \( t \) by \( t + i t^{\gamma} \) in (\ref{fracsop}) leads to the Schr\"odinger operator with complex time  
\[
P_{\gamma} f(x,t):=\frac{1}{(2\pi)^{d}} \int_{\mathbb{R}^{d}} e^{i \left(x\cdot\xi+ t|\xi|^{2}\right)} e^{-t^{\gamma}|\xi|^{2}} \hat{f}(\xi) \, d\xi.
\]  
This problem was first introduced by Sj\"olin \cite{MR2549801}, who showed that for \( 0 < \gamma \leq 1 \) and $d=1$, the condition \( f \in L^{2} \) is optimal, among other related results. Pointwise convergence for such operators was also examined by Sj\"olin and Soria \cite{MR2827425}. Subsequently, Bailey \cite{MR3047427} established that, in one dimension, the sharp regularity requirement is  $s > \min\left\{\frac{1}{4},\frac{1}{2}\left(1-\frac{1}{\gamma}\right)^{+} \right\}$. Some related problems have also been investigated, such as the dimension of divergence sets and the pointwise convergence of Schr\"odinger means for initial data in the Sobolev space \(W^{s,p}(\mathbb{R})\); see \cite{MR4222395,Pan-Sun} for further details.

In this paper, we investigate the convergence properties of $P_{\gamma} f(x,t)$ in higher dimensions. Here is the main result of this paper.
\begin{theorem}\label{theorem1}
    Let \( d\geq 2 \) and \(\gamma >0\). Then
    \begin{equation}\label{pointwise convergence in rd}
        \lim _{t \rightarrow 0} P_{ \gamma} f(x,t)=f(x) \quad \text{a.e. } x \in \mathbb{R}^d ,\quad \forall f \in H^{s}\left(\mathbb{R}^d\right)
    \end{equation}
    whenever  \( s>s_{0}=\min \left\{ \frac{d}{2(d+1)},\frac{d}{d+1}\left(1-\frac{1}{\gamma}\right)^{+} \right\} \). Conversely, (\ref{pointwise convergence in rd}) fails whenever $s<s_0$.
\end{theorem}

\begin{remark}
	When $0<\gamma \leq 1$ and $d=1$, $f\in L^2(\mathbb{R})$ is sufficient to ensure the pointwise convergence, see \cite{MR2549801}. When $1<\gamma < 2$, \( \min \left\{ \frac{d}{2(d+1)},\frac{d}{d+1}\left(1-\frac{1}{\gamma}\right)^{+} \right\}=\frac{d}{d+1}\left(1-\frac{1}{\gamma}\right) \), which indicates that the decay of $e^{-t^\gamma |\xi|^2}$ allows us to relax the regularity requirements.
	\end{remark}

		In the following parts, we will prove upper bounds for maximal functions in Section \ref{chap2}. Necessary conditions for convergence are shown in Section \ref{chap3}.\\
		
		\textbf{Notation.} Throughout this article,  we use $A \lesssim B$  to represent there exists a constant $C$, which does not depend on \( A \) and \( B \) such that $A \leq C B$. 
	  We write $A \gtrsim B$  to mean $B\lesssim A$. 
		We use \( A \sim B \) to mean that \( A \) and \( B \) are comparable, i.e. $A \lesssim B$  and $A \gtrsim B$.  
		We write $A \lesssim_{\alpha} B$  to mean that there exists a constant $C$ depending on variable $\alpha$ such that $A \leq C B$.
		We write  supp \( \hat{f} \subset\left\{\xi:|\xi| \sim R \right\} \) to mean 
		supp \( \hat{f} \subset\left\{\xi: \frac{R}{2}\le |\xi| \le 2R \right\} \) 
		and we will always assume $R\gg1$. 
		We use $C_{X}$ to denote a constant that depends on $X$, where $X$ is a variable.
		$\text {Let } \beta=\left(\beta_{1}, \ldots, \beta_{d}\right) \text { is a multiindex of order }|\beta|=\beta_{1}+\cdots+\beta_{d}=k$.
 We use $b^{+}$ to mean $\max (b,0)$.
 We denote by \(\lfloor a \rfloor\) the greatest integer less than or equal to \(a\) (the floor function), and by \(\lceil a \rceil\) the smallest integer greater than or equal to \(a\) (the ceiling function).
		
		\section{Proof of upper bound for maximal functions}\label{chap2}
		\numberwithin{equation}{section}
		
         		Via Littlewood–Paley decomposition and a standard smoothing argument, Theorem \ref{theorem1} can be reduced to the following maximal estimate.
		
	\begin{proposition}\label{proposition1}
			Let $d\geq 2$, $\gamma>0$ and $R\geq 1$. For any $\varepsilon>0$, we have
\begin{equation}\label{proposition1es}
				\left\|\sup _{0<t<1}\left|P_{ \gamma} f(x, t)\right|\right\|_{L^{2}(B^d(0,1))} \lesssim_{\varepsilon}
				R^{\min \left\{\frac{d}{2(d+1)},\frac{d}{d+1}(1-\frac{1}{\gamma})^{+}\right\}+\varepsilon}\|f\|_{2},
			\end{equation}
whenever \( \hat{f} \) is supported in \(\left\{\xi:|\xi|\sim R \right\} \). 
	\end{proposition}

Firstly, combining with temporal localization method in \cite{MR2970037}, we rewrite an estimate from \cite{MR4899579} as a lemma.
		\begin{lemma}\label{lemma1}
			Let $d\ge 2, J=(0,|J|)\subset[0,1] $ is an interval .  For any \( \varepsilon>0 \), we have
			\begin{equation}\label{3.1}
				\left\|\sup _{t \in J }\left|e^{i t(- \Delta)} f\right|\right\|_{L^{2}(B(0,1))} \lesssim_{\varepsilon} \begin{cases}
				\left(1+R^{\frac{d}{d+1}+\varepsilon} |J|^{\frac{d}{2(d+1)}}\right)\|f\|_{L^{2}\left(\mathbb{R}^{d}\right)}  & |J|\le R^{-1} \\
				R^{\frac{d}{2(d+1)}+\varepsilon} \|f\|_{L^{2}\left(\mathbb{R}^{d}\right)} &  |J|>  R^{-1} \\
				\end{cases}
			\end{equation}
		whenever \( \hat{f} \) is supported in \(\left\{\xi:|\xi|\sim R \right\} \). 		\end{lemma}
		We will use {Lemma \ref{lemma1}} to prove Proposition \ref{proposition1}.		 	
		\begin{proof} 
It is obvious that 
\[
\left\|\sup _{0<t<1}\left|P_{ \gamma} f(x, t)\right|\right\|_{L^{2}(B^d(0,1))}\leq \left\|\sup _{0<t\leq R^{-\frac{2}{\gamma}+\varepsilon}}\left|P_{ \gamma} f(x, t)\right|\right\|_{L^{2}(B^d(0,1))}+\left\|\sup _{R^{-\frac{2}{\gamma}+\varepsilon}<t<1}\left|P_{ \gamma} f(x, t)\right|\right\|_{L^{2}(B^d(0,1))}.
\]
		We first handle the second term, which is easier. Notice that
\[
\begin{aligned}
\sup _{R^{-\frac{2}{\gamma}+\varepsilon}<t<1}\left|P_{ \gamma} f(x, t)\right|&\leq \sup _{R^{-\frac{2}{\gamma}+\varepsilon}<t<1}\int_{|\xi|\sim R}  e^{-
	t^{\gamma}|\xi|^{2}}\phi(\frac{\xi}{R})|\hat{f}(\xi)| d \xi\\
&\lesssim e^{-R^{\varepsilon}}\int_{\mathbb{R}^d} \phi(\frac{\xi}{R})|\hat{f}(\xi)| d \xi\\
&\lesssim e^{-R^{\varepsilon}}\left(\int_{\mathbb{R}^d} \phi(\frac{\xi}{R})^2 d \xi\right)^{\frac{1}{2}}\|f\|_{L^{2}\left(\mathbb{R}^{d}\right)}\\
&\lesssim e^{-R^{\varepsilon}}R^{\frac{d}{2}}\|f\|_{L^{2}\left(\mathbb{R}^{d}\right)}
\end{aligned}
\]
		where \(\phi\) is a smooth, radial bump function satisfying $\phi(\xi) \equiv 1$  for $   \left\{ \xi :  |\xi| \sim 1 \right\}$ and $\operatorname{supp} \phi \subset \left\{ \xi : \frac{1}{3} \le |\xi| \le 3 \right\}$.     Then use the fact $e^{-y}\lesssim_\beta y^{-\beta}$ for any $\beta>0$, we can get
\[
\left\|\sup _{R^{-\frac{2}{\gamma}+\varepsilon}<t<1}\left|P_{ \gamma} f(x, t)\right|\right\|_{L^{2}(B^d(0,1))} \lesssim_{d, \varepsilon} \|f\|_{L^{2}\left(\mathbb{R}^{d}\right)}.
\]

Next, we handle the first term. We write
\begin{equation}\label{pgammapre}
\begin{aligned}
P_{\gamma} f(x,t)&=\frac{1}{(2\pi)^d}\int_{\mathbb{R}^d} e^{i \left(x\cdot\xi+ t|\xi|^{2}\right)} e^{-t^{\gamma}|\xi|^2}\phi(\frac{\xi}{R})\hat{f}(\xi) d \xi\\
&=\frac{R^d}{(2\pi)^d}\int_{\mathbb{R}^d} e^{i \left(Rx\cdot\xi+ R^2 t|\xi|^{2}\right)} e^{\Phi(t,\xi)}\phi(\xi)\hat{f}(R\xi) d \xi
\end{aligned}
\end{equation}
where \(\Phi(t,\xi) = -t^{\gamma}R^2|\xi|^2\). We consider \(\phi(\xi) e^{\Phi(t,\xi)}\) as a smooth function on the torus \(\mathbb{T}^d = [-\pi, \pi]^d\) via periodic extension; note that this extension is smooth because \(\operatorname{supp}\phi \subset \{\xi:\frac{1}{3}\le|\xi|\le 3\}\) is compact and well inside \((-\pi,\pi)^d\). Its Fourier series expansion is
\[
\phi(\xi) e^{\Phi(t,\xi)} = \sum_{l \in \mathbb{Z}^d} C_{l}(t) \, e^{i \xi \cdot l},
\]
where
\[
C_{l}(t) = \frac{1}{(2\pi)^d} \int_{[-\pi, \pi]^d} \phi(\xi) e^{\Phi(t,\xi)} \, e^{-i \xi \cdot l} \, d\xi.
\]
For \(0 < t \le R^{-\frac{2}{\gamma}+\varepsilon}\) and \(\xi\in\operatorname{supp}\phi\), we have the derivative estimate \(|\partial_\xi^\beta \Phi(t,\xi)| \lesssim R^{\gamma\varepsilon}\). Consequently, repeated integration by parts yields the following uniform decay estimate for the Fourier coefficients: there exists a constant \(M_d>0\) (depending only on the dimension \(d\) and the bump function \(\phi\)) such that
\[
|C_{l}(t)| \lesssim_d \frac{R^{M_d \gamma\varepsilon}}{(1+|l|)^{d+1}}, \qquad \forall\, l \in \mathbb{Z}^d.
\]
Returning to (\ref{pgammapre}) and inserting the Fourier expansion, we obtain
\[
\begin{aligned}
P_{\gamma} f(x,t) &= {\left(\frac{R}{2\pi}\right)}^{d} \int_{\frac{1}{2}\le|\xi|\le 2 } \sum_{l} C_{l}( t) e^{i \xi \cdot l} e^{i Rx\cdot\xi+R^2 t|\xi|^{2}} \hat{f}\left(R \xi\right) d \xi \\
&={\left(\frac{R}{2\pi}\right)}^{d} \sum_{l} \int_{\frac{1}{2}\le|\xi|\le 2} C_{l}( t) e^{i \xi \cdot l} e^{i Rx\cdot\xi+R^2 t|\xi|^{2}} \hat{f}\left(R \xi\right) d \xi \\
&={\left(\frac{1}{2\pi}\right)}^{d} \sum_{l} \int_{\frac{R}{2}\le|\xi|\le 2 R} C_{l}( t) e^{i \frac{\xi}{ R} \cdot l} e^{i x\cdot\xi+ t|\xi|^{2}} \hat{f}\left(\xi\right) d \xi .
\end{aligned}
\]
Define \(\hat{g_l}(\xi)=e^{i \frac{\xi}{ R }\cdot l} \hat{f}(\xi)\). By Plancherel's identity, \(\|g_l\|_{L^2} = \|f\|_{L^2}\). Using the bound for \(|C_l(t)|\) and the triangle inequality, we estimate
\[
\begin{aligned}
&\left\|\sup _{0<t\leq R^{-\frac{2}{\gamma}+\varepsilon}}\left|P_{\gamma} f(\cdot, t)\right|\right\|_{L^{2}(B(0,1))} \\
& \leq \sum_{l\in\mathbb{Z}^d} \sup_{t} |C_l(t)| \;
\left\|\sup _{0<t\leq R^{-\frac{2}{\gamma}+\varepsilon}}
\left|\int_{\frac{R}{2}\le|\xi|\le 2 R} e^{i x\cdot \xi} e^{i|\xi|^{2} t} \hat{g_l}(\xi) d \xi \right|\right\|_{L^{2}(B(0,1))} \\
& \lesssim_d R^{M_d\gamma\varepsilon} \sum_{l\in\mathbb{Z}^d} \frac{1}{(1+|l|)^{d+1}} \;
\left\|\sup _{0<t\leq R^{-\frac{2}{\gamma}+\varepsilon}} \left| e^{i t (-\Delta)} g_l \right|\right\|_{L^{2}(B(0,1))}.
\end{aligned}
\]
Since \(\sum_{l} (1+|l|)^{-(d+1)} < \infty\) and each \(g_l\) has Fourier support in \(\{\xi: |\xi|\sim R\}\), we can apply Lemma \ref{lemma1} uniformly to complete the estimate.
	\end{proof}
\begin{remark}
Through the proof steps, it is not difficult to see that the condition $0<t<1$ in Proposition \ref{proposition1es} can be strengthened to $t>0$. Based on the proof above, we can also establish an analogue of Lemma \ref{lemma1} for the maximal estimate associated with the Schrödinger operator with complex time, thereby addressing the problems of sequential convergence and convergence rate.
\end{remark}
\section{Necessity}\label{chap3}
We employ arguments from the Niki\v{s}hin-Stein theory to prove the necessity part of Theorem \ref{theorem1}.
\begin{proposition}\label{prop:necessity}
Let \( d \geq 2 \) and \( \gamma > 0 \). If the maximal estimate
\[
\left\|\sup_{0<t<1} \left|P_\gamma f(\cdot, t)\right|\right\|_{L^{2}(B(0,1))} \lesssim \|f\|_{H^s(\mathbb{R}^d)}
\]
holds, then it is necessary that
\[
s \geq \min \left\{ \frac{d}{d+1}\left(1-\frac{1}{\gamma}\right)^{+},\,\frac{d}{2(d+1)}\right\}.
\]
\end{proposition}

The idea of the counterexample originates from \cite{MR3574661}, and a more detailed explanation can be found in \cite{MR4186521}. For the reader's convenience, we follow the notations from \cite{MR3574661} and \cite{MR4186521}, but the parameters in the proof have been adjusted. To accommodate the nature of complex time, we omit the “Removal of the quadratic phase” part in \cite{MR4186521}. This omission allows us to circumvent the estimation of the Weyl‑type sum with a decay factor, and we proceed by applying Abel's summation formula directly, thereby avoiding the associated technical complications.

We will use the following four lemmas. Lemma \ref{abelsum} is a continuous partial summation formula, which will help us separate the error term from the main sum. Lemma \ref{weylsum} provides an estimate for quadratic Weyl sums, which is used to handle the incomplete Gauss sums (i.e., the remainder terms). Lemma \ref{Gausssum} gives an exact value for a specific Gauss sum, and we will apply it to compute the main term in the proof. Lemma \ref{Vitali} can be proved via the Vitali covering lemma; it is employed to obtain a lower bound for the measure of the new set under scaling. The proofs of these lemmas are omitted, since they can be found in \cite{MR4186521} and are relatively standard.

\begin{lemma}[Continuous Abel summation]\label{abelsum}
Let \(\{a_n\}_{n\in\mathbb{Z}}\) be a sequence of complex numbers and let \(h: \mathbb{R}\to\mathbb{C}\) be a continuously differentiable function. For any integers \(M, N \ge 0\), define the right-continuous step function
\[
A(u) = \sum_{n=M}^{\lfloor u \rfloor} a_n = \sum_{\substack{M \le n \le u \\ n \in \mathbb{Z}}} a_n, \qquad (u \ge M).
\]
 Then the following identity holds:
\[
\sum_{n=M}^{M+N} a_n\, h(n) = A(M+N)\, h(M+N) - \int_{M}^{M+N} A(u)\, h'(u)  \, du.
\]
\end{lemma}
	\begin{lemma}[Quadratic Weyl sum estimate]\label{weylsum}	
Let \( f(x) = \alpha x^{2} + \beta x \), where \( \alpha, \beta \in \mathbb{R} \). Suppose there exist \( a \in \mathbb{Z} \) and \( q \in \mathbb{N}^{+} \) such that  \( (q, a) = 1 \) and  the Diophantine condition \( \bigl| \alpha - \dfrac{a}{q} \bigr| \leq \dfrac{1}{q^{2}} \). Then there exists a constant \( C_{0} > 0 \), independent of \( \alpha, \beta, a, q \), such that for any \( M \in \mathbb{Z} \) and \( N \in \mathbb{N}^{+} \),
\[
\left|\sum_{M \leq n<M+N} e^{2 \pi i f(n)}\right| \leq C_{0}\left(\frac{N}{  q^{\frac{1}{2}}}+q^{\frac{1}{2}}\right)(\log q)^{\frac{1}{2}} .
\]
\end{lemma}	
\begin{lemma}[Gauss sum evaluation]\label{Gausssum}
Let \(a, b \in \mathbb{Z}\) and \(q \in \mathbb{N}\). Define the quadratic Gauss sum
\[
G(a, b ; q) := \sum_{\ell=1}^{q} e^{i\Bigl( 2\pi \ell \frac{b}{q} + 2\pi \ell^{2} \frac{a}{q} \Bigr)}.
\]
Suppose \((a, q) = 1\). If \(q \equiv 0 \pmod 4\) and \(b \equiv 0 \pmod 2\), then
\[
|G(a, b ; q)| = (2q)^{\frac12}.
\]
\end{lemma}
	\begin{lemma}[Lower bound for scaled unions]\label{Vitali}
Let \(\{B_j\}\) be finitely many cubes in \(\mathbb{R}^{d-1}\). For a constant \(0<c<1\), denote by \(B_j^*\) the cube with the same center as \(B_j\) and with side length scaled by \(c\). Then
\[
\Bigl| \bigcup_j B_j^* \Bigr| \ge c^{d-1}\, 3^{1-d} \Bigl| \bigcup_j B_j \Bigr|.
\]
\end{lemma}

	\begin{proof}	
		For $j=1,2,\cdots, d$, let \( \varphi_j: \mathbb{R} \rightarrow \mathbb{R}_{+} \) are standard bump functions  satisfying  supp \( {\varphi}_j \subset[-1,1] \) and \(\int_{\mathbb{R}} \varphi_j(\xi_j)d\xi_j =1\). Set $R$ be a large constant.  We divide the proof of necessity into three cases according to the value of \(\gamma\).

\noindent\textbf{Case 1:} \(0<\gamma\leq 1\).  
Define  
\[
\hat{g}(\xi)=\prod_{j=1}^{d}\frac{1}{R}\,
\varphi\!\left(\frac{\xi_j}{R}\right).
\]
Choosing \(t=0\) we have \(|P_\gamma g(x,0)|\sim 1\) for all \(x\in B^{d}(0,\frac{1}{1000R})\), which forces \(s\geq 0\).  

\noindent\textbf{Case 2:} \(\gamma>2\).  
Observe that a counterexample constructed for the critical exponent in the case \(\gamma=2\) remains a counterexample for every \(\gamma>2\).  
Thus it suffices to treat the case \(\gamma=2\), which is already covered in the next case.

\noindent\textbf{Case 3:} \(1<\gamma\leq 2\).  
This is the most delicate range and will occupy the rest of the proof. 

 We write \( x=\left(x_{1}, \ldots, x_{d}\right)=\left(x_{1}, x^{\prime}\right) \in B^d(0,1) \subset \mathbb{R}^{d} \) and  \( \xi=\left(\xi_{1}, \ldots, \xi_{d}\right)=\left(\xi_{1}, \xi^{\prime}\right) \in  \mathbb{R}^{d} \).  Set \( D=R^{\frac{d+\gamma}{ 2(d+1)}} \), and define
\[
\hat{f}(\xi)=\frac{1}{R^{\frac{1}{2}}}\varphi_1\left(\frac{\xi_1-R^{\frac{\gamma}{2}}}{R^{\frac{1}{2}} }\right)  \prod_{j=2}^{d}\left(\sum_{\frac{R^{\frac{\gamma}{2}}}{ D}\leq \ell_{j}<\frac{2R^{\frac{\gamma}{2}}}{D}} 
\varphi_j\left( \xi_j -D \ell_{j}  \right) \right),
\]
where \( \ell=\left(\ell_{2}, \ldots, \ell_{d}\right) \in \mathbb{Z}^{d-1} \). Notice that \(\operatorname{supp}\hat{f}\subset\{|\xi|\sim R^{\frac{\gamma}{2}}\}\),
hence one easily obtains
\begin{equation}\label{2.2}
\|f\|_{H^s} \sim R^{-\frac{1}{4}}\left(\frac{R^{\frac{\gamma}{2}}}{D}\right)^{\frac{d-1}{2}}R^{\frac{\gamma s}{2}}.
\end{equation}
Changing variables shows
\[
\begin{aligned}
P_{\gamma} f(x,t)&=\frac{1}{(2\pi)^d}\int_{\mathbb{R}^d} e^{i x\cdot \xi} e^{it |\xi|^2}     e^{-t^{\gamma} |\xi|^2}   \varphi_1\left(\frac{\xi_1-R^{\frac{\gamma}{2}}}{R^{\frac{1}{2}} }\right)  \prod_{j=2}^{d}\left(\sum_{\frac{R^{\frac{\gamma}{2}}}{ D}\leq \ell_{j}<\frac{2R^{\frac{\gamma}{2}}}{D}} 
\varphi_j\left( \xi_j -D \ell_{j} \right) \right)d\xi\\
=&\int_{\mathbb{R}^d} \prod_{j=1}^{d}\frac{1}{2\pi}\varphi_j(\xi_j) \\
&\times \left\{\sum_{\ell} e^{i\left(\left(R^{\frac{\gamma}{2}}+\xi_1 R^{\frac{1}{2}}\right) x_{1}+\left(\xi^{\prime}+D \ell\right) \cdot x^{\prime}+\left(R^{\frac{\gamma}{2}}+\xi_1 R^{\frac{1}{2}}\right)^{2} t+\left|\xi^{\prime}+D \ell\right|^{2} t\right)}e^{-\left(\left(R^{\frac{\gamma}{2}}+\xi_1 R^{\frac{1}{2}}\right)^{2} t^{\gamma}+\left|\xi^{\prime}+D \ell\right|^{2} t^{\gamma}\right)}\right\} d \xi.
\end{aligned}
\]
Write
\[
\begin{aligned}
(2\pi)^{d}|P_{\gamma} f(x,t)|&=\left| \int_{\mathbb{R}}\varphi_1(\xi_1) e^{i\left(\xi_1 R^{\frac{1}{2}}\left(x_1+2R^{\frac{\gamma}{2}}t\right)+ R\xi_1^2 t\right)} e^{-\left(R^{\frac{\gamma}{2}}+\xi_1 R^{\frac{1}{2}}\right)^{2} t^{\gamma}}d\xi_1 \right|\\
&\times \prod_{j=2}^{d} \left|  \sum_{\ell_j}e^{i\left( D \ell_j \cdot x_j +D^2\left| \ell_j\right|^{2} t\right)}     \left( \int_{\mathbb{R}}\varphi_j(\xi_j)e^{i\left( \xi_j(x_j+2Dt \ell_j )+\xi_j^2 t\right)}e^{-\left(\xi_j+D \ell_j\right)^{2} t^{\gamma}}  d \xi_j \right) \right|\\
&:=|I_1|\times  \prod_{j=2}^{d} | I_j|.
\end{aligned}
\]
Choose \(\displaystyle t=-\frac{x_{1}}{2 R^{\frac{\gamma}{2}}}+\tau \) with 
\(-c_1R^{\frac{\gamma}{2}-1}<x_1<-\frac{c_1}{2}R^{\frac{\gamma}{2}-1}\) and  
\(|\tau|<c_2R^{-\frac{\gamma+1}{2}}\), where the constants satisfy  
\(c_2 < \frac{c_1}{2} < \frac{c_0}{4}\) and \(c_0 \in \bigl(0,\frac{1}{2^{d+1}}\bigr)\). 
Then one can ensure that
\[
\bigl|\xi_1 R^{\frac{1}{2}}\bigl(x_1+2R^{\frac{\gamma}{2}}t\bigr)\bigr|
+ \bigl|\bigl(R^{\frac{\gamma}{2}}+\xi_1 R^{\frac{1}{2}}\bigr)^{2} t^{\gamma}\bigr|
\]
is sufficiently small, so that
\begin{equation}\label{I_1}
|I_1|>1-c_0.
\end{equation}

Next, we handle \( \displaystyle\prod_{j=2}^{d} I_j\).  
For \(\frac{R^{\frac{\gamma}{2}}}{D}<u\leq\big\lceil\frac{2R^{\frac{\gamma}{2}}}{D}\big\rceil\) and \(2\leq j\leq d\), define
\[
S_{j}(u)=S_{j}(x_{j},t;u):=
\sum_{\ell_{j}\in\mathbb{Z}:\,
      \frac{R^{\frac{\gamma}{2}}}{D}\leq\ell_{j}<u}
      e^{i\bigl(D\ell_{j}x_{j}+D^{2}\ell_{j}^{2}t\bigr)}.
\]
Note that \(S_{j}(u)=S_{j}\bigl(\lceil u\rceil\bigr)\). Define also
\[
S(x',t;u):=
\sum_{\substack{\ell'\in\mathbb{Z}^{d-1}\\
                \frac{R^{\frac{\gamma}{2}}}{D}\leq\ell_{j}<u}}
      e^{i\bigl(D\ell'\cdot x'+D^{2}|\ell'|^{2}t\bigr)},
\]
so that \(S(x',t;u)=S\bigl(x',t;\lceil u\rceil\bigr)\) and
\begin{equation}\label{prodinS}
S(x',t;u)=\prod_{j=2}^{d}S_{j}(u).
\end{equation}
For convenience, we set
\[
2R^{\frac{\gamma}{2}\prime}=D\Bigl(\Big\lceil\frac{2R^{\frac{\gamma}{2}}}{D}\Big\rceil-1\Bigr).
\]
For each \(I_j\), we apply Lemma \ref{abelsum} (the continuous partial summation formula) with
\[
a_{\ell}=e^{i(D\ell x_{j}+D^{2}\ell^{2}t)},\qquad 
h(\ell)=\int_{\mathbb{R}}\varphi_{j}(\xi_{j})
        e^{i[\xi_{j}(x_{j}+2Dt\ell)+\xi_{j}^{2}t]}
        e^{-(\xi_{j}+D\ell)^{2}t^{\gamma}}\,d\xi_{j},
\]
where \(\ell\) runs from \(M:=\big\lceil\frac{R^{\frac{\gamma}{2}}}{D}\big\rceil\) to 
\(L:=\big\lceil\frac{2R^{\frac{\gamma}{2}}}{D}\big\rceil-1\). Then
\begin{equation}\label{estimateonMj}
\begin{aligned}
I_{j}&=
\sum_{\ell=M}^{L}e^{i(D\ell x_{j}+D^{2}\ell^{2}t)}
      \int_{\mathbb{R}}\varphi_{j}(\xi_{j})
        e^{i[\xi_{j}(x_{j}+2Dt\ell)+\xi_{j}^{2}t]}
        e^{-(\xi_{j}+D\ell)^{2}t^{\gamma}}\,d\xi_{j}   \\
&=S_{j}\!\left(\frac{2R^{\frac{\gamma}{2}}}{D}\right)
   \int_{\mathbb{R}}\varphi_{j}(\xi_{j})
        e^{i[\xi_{j}(x_{j}+4tR^{\frac{\gamma}{2}\prime})+\xi_{j}^{2}t]}
        e^{-(\xi_{j}+2R^{\frac{\gamma}{2}\prime})^{2}t^{\gamma}}\,d\xi_{j} \\
&\quad -\int_{M}^{L}S_{j}\bigl(\lceil u\rceil\bigr)\,
      \Bigl(\int_{\mathbb{R}}\varphi_{j}(\xi_{j})2D
        \bigl(i\xi_{j}t-\xi_{j}t^{\gamma}-Du t^{\gamma}\bigr)
        e^{i[\xi_{j}(x_{j}+2Dt u)+\xi_{j}^{2}t]}
        e^{-(\xi_{j}+Du)^{2}t^{\gamma}}\,d\xi_{j}\Bigr)du   \\
&=S_{j}\!\left(\frac{2R^{\frac{\gamma}{2}}}{D}\right)
   \int_{\mathbb{R}}\varphi_{j}(\xi_{j})
        e^{i[\xi_{j}(x_{j}+4tR^{\frac{\gamma}{2}\prime})+\xi_{j}^{2}t]}
        e^{-(\xi_{j}+2R^{\frac{\gamma}{2}\prime})^{2}t^{\gamma}}\,d\xi_{j}
   \;+\;E_{j}{(1)} .
\end{aligned}
\end{equation}
Here we have set
\[
E_{j}{(1)}:=-\int_{M}^{L}S_{j}\bigl(\lceil u\rceil\bigr)\,h'(u)\,du.
\]
An elementary estimate then yields
\begin{equation}\label{ej2}
\bigl|E_{j}{(1)}\bigr|
\leq 4\Bigl(R^{\frac{\gamma}{2}}t+(tR)^{\gamma}\Bigr)
      \sup_{\frac{R^{\frac{\gamma}{2}}}{D}\leq u<\frac{2R^{\frac{\gamma}{2}}}{D}}
      \bigl|S_{j}(u)\bigr| .
\end{equation}
Combining (\ref{estimateonMj}) and (\ref{ej2}), we can get 
\begin{equation}\label{final2}
\displaystyle \prod_{j=2}^{d}  I_j=S(x^{\prime},t;\frac{2R^{\frac{\gamma}{2}}}{D})\prod_{j=2}^{d} \left( \int_{\mathbb{R}}\varphi_j(\xi_j)e^{i \left[\xi_j(x_j+4t R^{\frac{\gamma}{2}\prime} )+\xi_j^2 t\right]}e^{-\left(\xi_j+2R^{\frac{\gamma}{2}\prime}\right)^{2} t^{\gamma}}  d \xi_j \right)+E(1)
\end{equation}
with $E(1)$ is a sum of $(2^{d-1}-1)$ parts which satisfies
\begin{equation}\label{estimateofe2}
|E(1)|\leq (2^{d-1}-1)(\text{RHS of } \ref{ej2}) \max\left\{ (\text{RHS of } \ref{ej2}),\left|S_{j}( \frac{2R^{\frac{\gamma}{2}}}{D})\right|\right\}^{d-2}.
\end{equation}

We now construct the set \(\Omega^*\)  containing the desired positions of \(x\). 
Starting from the set \(\Omega\) defined below, we will perform appropriate translations and scalings  to obtain \(\Omega^*\).  For every \(x\in\Omega^*\) we will show that the product \(\displaystyle \prod_{j=2}^{d}|I_j|\) admits a lower bound.  More precisely,
\[
\prod_{j=2}^{d}|I_j|\geq (1-c_0)^{d-1} 
\Bigl|S\Bigl(x^{\prime},t;\frac{2R^{\frac{\gamma}{2}}}{D}\Bigr)\Bigr|-|E{(1)}|,
\]
where \(E{(1)}\) is an error term that will be shown to be negligible. Let 
\[
c_{3}<\min\Bigl\{\frac{c_2}{4},\frac{1}{2\pi}\Bigr\},\qquad 
\mu_{0}:=\frac{1}{(4\pi)^{d}},\qquad 
c_{4}<\frac{1}{2},
\]
and set \(Q=R^{\frac{(\gamma-1)(d-1)}{2(d+1)}}\). Write \(y=(y_1,y')\). Consider the set
\[
y \in \Omega:=\bigcup_{\substack{4\mu_{0} Q \leq q \leq 4 Q \\ q\equiv0\pmod4}} 
        \bigcup_{\substack{1\leq a_{1}\leq q \\ (a_{1},q)=1}}
        \bigcup_{\substack{2\leq a_{2},\dots,a_{d}\leq \frac{q}{2} \\ a_{j}\equiv0\pmod2}}
        \Bigl(\prod_{i=1}^{d}\Bigl[\frac{2\pi a_{i}}{q}-A_{i},\frac{2\pi a_{i}}{q}+A_{i}\Bigr] 
        \bmod 2\pi\mathbb{T}^{d}\Bigr),
\]
where 
\[
A_{1}=\frac{\pi c_{3}}{4Q},\qquad 
A_{2}=\cdots=A_{d}=\frac{\pi c_{4}}{\mu_{0}Q^{\frac{d}{d-1}}}.
\]
We claim that for any \(\varepsilon_{0}>0\), there exists a constant \(c_{\varepsilon_{0}}>0\) such that
\[
|\Omega| \geq c_{\varepsilon_{0}}\,2^{-d}\,3^{1-d}\,c_4^{d-1}\,Q^{-\varepsilon_{0}}.
\]
For the first coordinate, note that \(c_3<\frac{1}{2\pi}\). Hence for each \(q\in[4\mu_{0}Q,4Q]\), the set
\[
\mathscr{V}_{1}(q):=\bigcup_{\substack{1\leq a_{1}\leq q \\ (a_{1},q)=1}}
          \Bigl\{y_{1}\in[0,2\pi]:\Bigl|y_{1}-\frac{2\pi a_{1}}{q}\Bigr|<\frac{\pi c_{3}}{4Q}\Bigr\}
\]
consists of disjoint intervals in \([0,2\pi]\). Using the Euler totient function, one obtains
\[
\min_{4\mu_{0}Q\leq q\leq4Q}\bigl|\mathscr{V}_{1}(q)\bigr| \geq c_{\varepsilon_{0}} Q^{-\varepsilon_{0}}.
\]
Now consider the remaining coordinates. Define
\[
\mathscr{V}_{2}:=\bigcup_{\substack{4\mu_{0}Q\leq q\leq4Q \\ q\equiv0\pmod4}} \mathscr{V}_{2}(q),
\]
where
\[
\mathscr{V}_{2}(q):=\bigcup_{\substack{2\leq a_{2},\dots,a_{d}\leq 2q \\ a_{j}\equiv0\pmod2}}
          \Bigl\{y'\in[0,2\pi]^{d-1}:\Bigl|y_{j}-\frac{2\pi a_{j}}{q}\Bigr|<\frac{\pi c_{4}}{\mu_{0}Q^{\frac{d}{d-1}}},\; j=2,\dots,d\Bigr\}.
\]
Since \(\displaystyle \Omega = \bigcup_{\substack{4\mu_{0}Q\leq q\leq4Q \\ q\equiv0\pmod4}} \bigl(\mathscr{V}_{1}(q)\times\mathscr{V}_{2}(q)\bigr)\), it suffices to prove
\[
\bigl|\mathscr{V}_{2}\bigr| \geq 2^{-d}\,3^{1-d}\,c_4^{d-1}.
\]
Via simultaneous Dirichlet's approximation,  we obtain
\[
\Bigl|\bigcup_{1\leq q\leq Q}
       \bigcup_{\substack{1\leq a_{j}\leq q \\ 2\leq j\leq d}}
       J(q;a_{2},\dots,a_{d})\Bigr| \ge 1,
\]
where
\[
J(q;a_{2},\dots,a_{d}):=\prod_{j=2}^{d}
      \Bigl[\frac{2\pi a_{j}}{q}-\frac{2\pi}{q Q^{\frac{1}{d-1}}},
             \frac{2\pi a_{j}}{q}+\frac{2\pi}{q Q^{\frac{1}{d-1}}}\Bigr].
\]
Recall that \(\mu_{0}=(4\pi)^{-d}\). Hence
\[
\sum_{1\leq q\leq \mu_{0}Q}
\sum_{\substack{1\leq a_{j}\leq q \\ 2\leq j\leq d}}
\bigl|J(q;a_{2},\dots,a_{d})\bigr|
\le (4\pi)^{d-1}\mu_{0} = \frac{1}{4\pi} < \frac{1}{2}.
\]
Consequently,
\[
\Bigl|\bigcup_{\mu_{0}Q\leq q\leq Q}
       \bigcup_{\substack{1\leq a_{j}\leq q \\ 2\leq j\leq d}}
       J(q;a_{2},\dots,a_{d})\Bigr| \ge \frac{1}{2}.
\]
Now we apply a scaling argument. Observe that the condition
\[
\frac{y'}{2} \in \prod_{j=2}^{d}
      \Bigl[\frac{2\pi(2a_{j})}{4q}-\frac{4\pi}{4q Q^{\frac{1}{d-1}}},
             \frac{2\pi(2a_{j})}{4q}+\frac{4\pi}{4q Q^{\frac{1}{d-1}}}\Bigr]
\]
is equivalent to setting \(a_{j}'=2a_{j}\) and \(q'=4q\). Then we have
\(4\mu_{0}Q\le q'\le4Q\), \(q'\equiv0\pmod4\), and \(2\le a_{j}'\le \frac{q'}{2}\) for \(2\le j\le d\). Moreover,
\[
\frac{y'}{2} \in \prod_{j=2}^{d}
      \Bigl[\frac{2\pi a_{j}'}{q'}-\frac{4\pi}{q' Q^{\frac{1}{d-1}}},
             \frac{2\pi a_{j}'}{q'}+\frac{4\pi}{q' Q^{\frac{1}{d-1}}}\Bigr].
\]
Therefore,
\[
\Bigl|\bigcup_{\substack{4\mu_{0}Q\le q'\le4Q \\ q'\equiv0\pmod4}}
       \bigcup_{\substack{2\le a_{2}',\dots,a_{d}'\le\frac{q'}{2} \\ a_{j}'\equiv0\pmod2}}
       \prod_{j=2}^{d}
       \Bigl[\frac{2\pi a_{j}'}{q'}-\frac{4\pi}{q' Q^{\frac{1}{d-1}}},
              \frac{2\pi a_{j}'}{q'}+\frac{4\pi}{q' Q^{\frac{1}{d-1}}}\Bigr]\Bigr|
   \ge \frac{1}{2^{d}}.
\]
Finally, note that \(q'\ge4\mu_{0}Q\) and \(\frac{4\pi}{q'Q^{\frac{1}{d-1}}}\le\frac{\pi c_{4}}{\mu_{0}Q^{\frac{d}{d-1}}}\) 
for our choice of \(c_{4}<\frac{1}{2}\). Applying Lemma~\ref{Vitali} yields the desired estimate
\[
\bigl|\mathscr{V}_{2}\bigr| \ge 2^{-d}\,3^{1-d}\,c_{4}^{d-1},
\]
which completes the proof of the claim.
Next, we determine the location of the point \(x\). For $j=2, \cdots, d$,  consider
\[
\begin{aligned}
x\in \Omega^{*}=&\left\{x \in\left[-c_{1}R^{\frac{\gamma}{2}-1},-\frac{c_{1}R^{\frac{\gamma}{2}-1}}{2}\right] \times\left[-c_{1}, c_{1}\right]^{d-1} \right. \\
&:\left. \exists y \in \Omega, \text { s.t. } y_{1} \equiv-\frac{D^{2}}{2 R^{\frac{\gamma}{2}}} x_{1}(\bmod 2 \pi), y_{j} \equiv D x_{j}(\bmod 2 \pi) \right\}.
\end{aligned}
\]
The region \(\Omega^{*}\) is constructed from \(\Omega\) by the following procedure.
Starting from  \(y \in \Omega\), we first make several periodic copies, 
and then scale each coordinate individually into the target interval
\[
\left[-c_{1}R^{\frac{\gamma}{2}-1},\,-\frac{c_{1}R^{\frac{\gamma}{2}-1}}{2}\right]
\times \bigl[-c_{1}, c_{1}\bigr]^{d-1}.
\]
To illustrate the idea, take the second coordinate as an example.  
Choose an integer \(M \gg 1\) such that \(M c_{1} \geq 2\pi\).  
Define the map \(\iota: \mathbb{R} \to [0,2\pi)\) by \(\iota(z)=\bar{z}\), 
where \(\bar{z}\equiv z \pmod{2\pi}\).  
For a set \(S_{0}\subset [0,2\pi)\), set
\[
S_{1}:=\iota^{-1}(S_{0}) \cap \bigl[-M c_{1}, M c_{1}\bigr).
\]
Clearly, 
\[
|S_{1}| \geq 2\Bigl\lfloor \frac{M c_{1}}{2\pi}\Bigr\rfloor \,|S_{0}|.
\]
Define the scaling map \(r:\mathbb{R}\to\mathbb{R}\) by \(r(z)=Mz\), and let 
\(S_{2}=r^{-1}(S_{1})\). Then \(S_{2}\subset [-c_{1},c_{1})\) and
\[
|S_{2}|=\frac{|S_{1}|}{M}
        \geq \frac{c_{1}}{2\pi}\,|S_{0}|.
\]
Now apply the same two steps to every coordinate.   For the first coordinate we use the scaling factor \(M_{1}= \dfrac{D^{2}}{2R^{\frac{\gamma}{2}}}\), and for the remaining coordinates \(j=2,\dots,d\) we use \(M_{j}=D\).  
A direct computation then gives
\[
|\Omega^{*}| \geq 
\frac{R^{\frac{\gamma}{2}-1}c_{1}^{d}}{4\,(2\pi)^{d}}\;|\Omega|
\geq \frac{c_{1}^{d}}{4\,(2\pi)^{d}}\,
c_{\varepsilon_{0}}\,2^{-d}3^{1-d}c_4^{d-1} \,Q^{-\varepsilon_{0}}\,R^{\frac{\gamma}{2}-1}.
\]

Our next goal is to prove when $x\in \Omega^{*}$, we can choose suitable $t$ satisfying
\[
\left|P_{\gamma} f(x,t_x)\right|\sim R^{\frac{(\gamma-1)(d-1)}{4}}.
\]
Since $ \displaystyle t=-\frac{x_{1}}{2 R^{\frac{\gamma}{2}}}+\tau $, by leveraging the variability of $\tau$, we can achieve the desired estimate. Specifically, set $s=D^2 \tau$ and $y_1+s=\frac{2\pi a_1}{q}=D^2 t (\bmod 2 \pi)$. Then there exists $q\in [4\mu_0 Q,4Q]$ such that
\[
\left|S(x^{\prime},t;\frac{2R^{\frac{\gamma}{2}}}{D})\right|=\left(\frac{\sqrt{2} R^{\frac{\gamma}{2}}}{D q^{\frac{1}{2}}}\right)^{d-1}+E(2)
\] 
with
\begin{equation}\label{estimateofe3}
|E(2)|\leq C_{3}\left(d, \delta_0, \mu_{0}\right)\left(c_{4}+ R^{-\delta_0}\right)\left(\frac{R^{\frac{\gamma}{2}}}{L Q^{\frac{1}{2}}}\right)^{d-1}.
\end{equation}
Here, $\delta_0$ is a small constant satisfying $\delta_0<\frac{\gamma-1}{4(d+1)}$. Recall
$
 \displaystyle S_{j}(u)=\sum_{\frac{R^{\frac{\gamma}{2}}}{D}  \leq \ell_{j}<u} e^{i\left(D \ell_{j} x_{j}+D^{2} \ell_{j}^{2} t\right)}
$
 and  define
\begin{equation}\label{tildes}
\tilde{S}_{j}(u):=\sum_{\frac{R^{\frac{\gamma}{2}}}{D}  \leq \ell_{j}<u} e^{i\left(\ell_{j}\frac{2 \pi a_{j} }{q}+\ell_{j}^{2}\left(y_{1}+s\right)\right)}.
\end{equation}
We write
\[
\begin{aligned}
\left|S(x^{\prime},t;\frac{2R^{\frac{\gamma}{2}}}{D})\right|&=\prod^{d}_{j=2}\left|S_{j}(\frac{2R^{\frac{\gamma}{2}}}{D})\right|\\
&=\prod^{d}_{j=2} \left(\frac{\sqrt{2}R^{\frac{\gamma}{2}}}{D q^{\frac{1}{2}}}+E_j(2)\right)
\end{aligned}
\]
with 
\begin{equation}\label{E_j(2)}
\begin{aligned}
|E_j(2)|&=\left|\left|S_{j}(\frac{2R^{\frac{\gamma}{2}}}{D})\right|-\left|\tilde{S}_{j}(\frac{2R^{\frac{\gamma}{2}}}{D})\right|+\left|\tilde{S}_{j}(\frac{2R^{\frac{\gamma}{2}}}{D})\right|-\frac{\sqrt{2}R^{\frac{\gamma}{2}}}{D q^{\frac{1}{2}}}\right|\\
&\leq \left|\left|S_{j}(\frac{2R^{\frac{\gamma}{2}}}{D})\right|-\left|\tilde{S}_{j}(\frac{2R^{\frac{\gamma}{2}}}{D})\right|\right|+\left|\left|\tilde{S}_{j}(\frac{2R^{\frac{\gamma}{2}}}{D})\right|-\frac{\sqrt{2}R^{\frac{\gamma}{2}}}{D q^{\frac{1}{2}}}\right|.
\end{aligned}
\end{equation}
Therefore, we now need to make the error term sufficiently small in order to prove (\ref{estimateofe3}). We first handle the second term. For \( \frac{R^{\frac{\gamma}{2}}}{D} \leq u \leq \frac{2R^{\frac{\gamma}{2}}}{D} \), Lemma \ref{Gausssum} allows us to split the sum in (\ref{tildes}) into  
\[
\biggl\lfloor \frac{\lceil u \rceil - \bigl\lceil \tfrac{R^{\frac{\gamma}{2}}}{D} \bigr\rceil}{q} \biggr\rfloor
\]  
complete Gauss sums, each of modulus \((2q)^{\frac{1}{2}}\), plus some leftover terms.  Via Lemma \ref{weylsum}, when \( 4 \mu_{0} Q \leq q \leq 4 Q \)  we obtain
\[
\begin{aligned}
\left|\left|\tilde{S}_{j}(u)\right|-\left\lfloor\frac{\lceil u\rceil-\lceil \frac{R^{\frac{\gamma}{2}}}{D} \rceil}{q}\right\rfloor \cdot(2 q)^{\frac{1}{2}}\right|&\leq \sup _{\substack{k \in \mathbb{N}_{+} \\ 1 \leq k<q}}\left| \sum_{\lceil \frac{R^{\frac{\gamma}{2}}}{D} \rceil \leq v<\lceil \frac{R^{\frac{\gamma}{2}}}{D} \rceil+k} e^{i\left(v\frac{2 \pi a_{j}} {q}+v^{2}\frac{2 \pi a_{1}} {q}\right)} \right| \\
&\leq  2 C_{0} q^{\frac{1}{2}}(\log q)^{\frac{1}{2}} .
\end{aligned}
\]
Replacing the floor function by its linear approximation \(\displaystyle \frac{u - \frac{R^{\frac{\gamma}{2}}}{D}}{q}\) and absorbing the resulting error
\begin{equation}\label{sju-Gausssum}
\begin{aligned}
\left|\left|\tilde{S}_{j}(u)\right|-\frac{\sqrt{2}(u-\frac{R^{\frac{\gamma}{2}}}{D} )}{q^{\frac{1}{2}}}\right| & \leq 2 C_{0} q^{\frac{1}{2}}(\log q)^{\frac{1}{2}}+2 \sqrt{2} q^{\frac{1}{2}} \\
&\leq\left(2 C_{0}+2\right) q^{\frac{1}{2}}(\log q)^{\frac{1}{2}}\\
&\leq\left(2 C_{0}+2\right) 2  Q^{\frac{1}{2}}(\log 4Q)^{\frac{1}{2}} \\
&\leq C_{\delta_0}\frac{R^{\frac{\gamma}{2}-\delta_0} }{ D Q^{\frac{1}{2}}}.
\end{aligned}
\end{equation}
For any \( u \in[\frac{R^{\frac{\gamma}{2}}}{D} ,  \frac{2R^{\frac{\gamma}{2}}}{D} ] \), (\ref{sju-Gausssum}) yields 
 \begin{equation}\label{sjutilede}
 \left|\tilde{S}_{j}(u)\right| \leq\left(\frac{\sqrt{2}}{2 \mu_{0}^{\frac{1}{2}}}+\frac{1}{10}\right) \frac{R^{\frac{\gamma}{2}}}{D Q^{\frac{1}{2}}}<(4 \pi)^{d} \frac{R^{\frac{\gamma}{2}}}{D Q^{\frac{1}{2}}}.
 \end{equation}
Next, we handle the first error term. Recall that 
\begin{equation}\label{2.5}
y_j=D x_j(\bmod 2 \pi), \quad y_1+s=\frac{2\pi a_1}{q}=D^2 t (\bmod 2 \pi)
\end{equation}
to get
\[
\left|S_{j}(u)-\tilde{S}_{j}(u)\right|=\sum_{\frac{R^{\frac{\gamma}{2}}}{D}  \leq \ell_{j}<u} e^{i\left(\ell_{j}\left(\frac{2 \pi a_{j} }{q}\right)+\ell_{j}^{2}\left(y_{1}+s\right)\right)}e^{i\left(\ell_{j}\left(y_{j}-\frac{2 \pi a_{j} }{q}\right)\right) }.
\]
Thus we write
\[
\begin{aligned}
\left|\left|S_{j}(u)\right|-\left|\tilde{S}_{j}(u)\right|\right|&\leq \left|S_{j}(u)-\tilde{S}_{j}(u)\right|\\
&\leq   \left| S_{j}(u)-e^{i\left(u \left(y_{j}-\frac{2 \pi a_{j} }{q}\right)\right) }\tilde{S}_{j}(u)  \right|+ \left| e^{i\left(u \left(y_{j}-\frac{2 \pi a_{j} }{q}\right)\right) }\tilde{S}_{j}(u) -\tilde{S}_{j}(u) \right|
\end{aligned}
\]
Using Lemma \ref{abelsum} with
\[
a_\ell =e^{i\left(\ell \left(\frac{2 \pi a_{j} }{q}\right)+\ell^{2}\left(y_{1}+s\right)\right)}, \quad h(\ell)= e^{i\left(\ell \left(y_{j}-\frac{2 \pi a_{j} }{q}\right)\right) }
\]
one obtain
\begin{equation}\label{sju-tildesju1}
\begin{aligned}
\left|S_{j}(u)-e^{i\left(u \left(y_{j}-\frac{2 \pi a_{j} }{q}\right)\right) }\tilde{S}_{j}(u)\right|&\leq  \sup _{v \in[0, \frac{R^{\frac{\gamma}{2}}}{D} ]}\left|\sum_{\frac{R^{\frac{\gamma}{2}}}{D}  \leq k \leq\lceil \frac{R^{\frac{\gamma}{2}}}{D} \rceil+\lfloor v\rfloor} e^{i\left(k\left(\frac{2 \pi a_{j} }{q}\right)+k^{2}\left(y_{1}+s\right)\right)} \right| \cdot\left|y_{j}-\frac{2 \pi a_{j} }{q}\right| \cdot(u-\frac{R^{\frac{\gamma}{2}}}{D} ) \\
&\leq   (4 \pi)^{d}  \frac{R^{\frac{\gamma}{2}}} { DQ^{\frac{1}{2}}}\frac{\pi c_{4}}{\mu_{0}Q^{\frac{d}{ d-1}}}\frac{R^{\frac{\gamma}{2}}} { D}\\
&\leq  4 c_{4} (4 \pi)^{2d} \frac{R^{\frac{\gamma}{2}}} { DQ^{\frac{1}{2}}},
\end{aligned}
\end{equation}
where we use the fact ${Q^{\frac{d}{ d-1}}}=\frac{R^{\frac{\gamma}{2}}}{D}  $ and $\frac{R^{\frac{\gamma}{2}}} { D}\leq u \leq \frac{2R^{\frac{\gamma}{2}}} { D}$.  In addition
\begin{equation}\label{sju-tildesju2}
\begin{aligned}
\left| e^{i\left(u \left(y_{j}-\frac{2 \pi a_{j} }{q}\right)\right) }\tilde{S}_{j}(u) -\tilde{S}_{j}(u) \right|&\leq u \left|y_{j}-\frac{2 \pi a_{j} }{q}\right|\cdot\left|\tilde{S}_{j}(u) \right|\\
& \leq  \frac{2R^{\frac{\gamma}{2}}} { D} \frac{\pi c_{4}}{\mu_{0}Q^{\frac{d}{ d-1}}}  (4 \pi)^{d} \frac{R^{\frac{\gamma}{2}}}{D Q^{\frac{1}{2}}}\\
&\leq 8 c_{4} (4 \pi)^{2d} \frac{R^{\frac{\gamma}{2}}} { DQ^{\frac{1}{2}}}.
\end{aligned}
\end{equation}
Combining (\ref{sju-tildesju1}) and (\ref{sju-tildesju2}) gives 
\begin{equation}\label{sju-tildesju}
\left|\left|S_{j}(u)\right|-\left|\tilde{S}_{j}(u)\right|\right|\leq 12 c_{4} (4 \pi)^{2d} \frac{R^{\frac{\gamma}{2}}} { DQ^{\frac{1}{2}}}.
\end{equation}
Using (\ref{sju-Gausssum}) and (\ref{sju-tildesju}) in (\ref{E_j(2)}), we have
\[
\begin{aligned}
|E_j(2)|&\leq C_{\delta_0}\frac{R^{\frac{\gamma}{2}-\delta_0} }{D Q^{\frac{1}{2}}}+12 c_{4}(4 \pi)^{2d} \frac{R^{\frac{\gamma}{2}}} { DQ^{\frac{1}{2}}}\\
&\leq \left[ C_{\delta_0} R^{-\delta_0} +12 c_{4} (4 \pi)^{2d}  \right]\frac{R^{\frac{\gamma}{2}}} { DQ^{\frac{1}{2}}}.
\end{aligned}
\]
$E(2)$ is a sum of $(2^{d-1}-1)$ part which satisfies
\[
|E(2)|\leq (2^{d-1}-1) \left\{\left[ C_{\delta_0} R^{-\delta_0} +12 c_{4} (4 \pi)^{2d}  \right]\frac{R^{\frac{\gamma}{2}}} { DQ^{\frac{1}{2}}} \right\} \left\{ \sqrt{2}(4 \pi)^{\frac{d}{2}} \frac{R^{\frac{\gamma}{2}}}{D Q^{\frac{1}{2}}} \right\}^{d-2},
\]
which finish the proof of (\ref{estimateofe3}). When $R$ is large and $c_4$ sufficiently small,  one can get
\begin{equation}\label{finalE(2)}
|E(2)|\leq 2^{-\frac{d+5}{2}}\left( \frac{R^{\frac{\gamma}{2}}}{D Q^{\frac{1}{2}}} \right)^{d-1}.
\end{equation}
We also need the upper bound for $S_j(u)$ to  bound $E(1)$.  Combining  (\ref{sjutilede}) and (\ref{sju-tildesju}) gives
\[
|{S}_{j}(u)|\leq 2 (4 \pi)^{d} \frac{R^{\frac{\gamma}{2}}}{D Q^{\frac{1}{2}}} .
\]
Thus combining  with the estimate (\ref{estimateofe2}), we have
\[
|E(1)| \leq 2^{d+1}\left(2 (4 \pi)^{d}\right)^{d-2} Rt \left(\frac{R^{\frac{\gamma}{2}}}{D Q^{\frac{1}{2}}}\right)^{d-1} .
\]
When $R$ is large and $c_1$ sufficiently small,  one can get
\begin{equation}\label{finalE(1+2)}
|E(1)| \leq 2^{-\frac{d+5}{2}}\left( \frac{R^{\frac{\gamma}{2}}}{D Q^{\frac{1}{2}}} \right)^{d-1}.
\end{equation}

Recall that 
\[
P_{\gamma} f(x,t)=\prod^{d}_{j=1} I_j
\]
with (\ref{I_1}) and
\[
\begin{aligned}
\prod^{d}_{j=2}| I_j|&\geq (1-c_0)^{d-1} \left| S(x^{\prime},t;\frac{2R^{\frac{\gamma}{2}}}{D})\right|-|E(1)|\\
&\geq  (1-c_0)^{d-1} \left(\frac{\sqrt{2} R^{\frac{\gamma}{2}}}{D q^{\frac{1}{2}}}\right)^{d-1}-|E(2)|-|E(1)|.
\end{aligned}
\]
Notice that for any \( 4 \mu_{0} Q \leq q \leq 4 Q \) 
\[
 \left(\frac{\sqrt{2} R^{\frac{\gamma}{2}}}{D q^{\frac{1}{2}}}\right)^{d-1}\geq 2^{\frac{1-d}{2}}  \left(\frac{R^{\frac{\gamma}{2}}}{D Q^{\frac{1}{2}}}\right)^{d-1}.
\]
Given that the error terms (\ref{finalE(2)}) and (\ref{finalE(1+2)}) are sufficiently small for \( x \in \Omega^{*} \) and appropriate \( t \), and noting (\ref{2.2}), it follows that
\[
\frac{\displaystyle\left\|\sup _{0<t<1} \left|P_\gamma f(x,t)\right|\right\|_{L^2}}{\|f\|_{H^s}} \gtrsim R^{\frac{(\gamma-1)(d-1)}{4}-\varepsilon_0} R^{\frac{\gamma-2}{4}}R^{\frac{1}{4}}\left(\frac{D}{R^{\frac{\gamma}{2}}}\right)^{\frac{d-1}{2}}R^{-\frac{\gamma s}{2}}=R^{\frac{d(\gamma-1)}{2(d+1)}-\frac{\gamma s}{2}-\varepsilon_0},
\]
which gives $s\geq  \frac{d}{d+1}\left(1-\frac{1}{\gamma}\right)+\frac{2\varepsilon_0}{\gamma}$ when $R$ is sufficiently large. Let $\varepsilon_0\to 0$, we get the desired estimate.

\end{proof}

		~\\

		Meng Wang, Department of Mathematics, Zhejiang University, Hangzhou 310058, China
		
		E-mail address: mathdreamcn@zju.edu.cn\\
		
		Zhichao Wang, Department of Mathematics, Zhejiang University, Hangzhou 310058,  China
		
		E-mail address: zhichaowang@zju.edu.cn

	\end{sloppypar}

\begin{thebibliography}{10}

\bibitem{MR3047427}
Andrew~D. Bailey.
\newblock Boundedness of maximal operators of {S}chr\"{o}dinger type with
  complex time.
\newblock {\em Rev. Mat. Iberoam.}, 29(2):531--546, 2013.

\bibitem{MR1315543}
Jean Bourgain.
\newblock Some new estimates on oscillatory integrals.
\newblock In {\em Essays on {F}ourier analysis in honor of {E}lias {M}. {S}tein
  ({P}rinceton, {NJ}, 1991)}, volume~42 of {\em Princeton Math. Ser.}, pages
  83--112. Princeton Univ. Press, Princeton, NJ, 1995.

\bibitem{MR3241836}
Jean Bourgain.
\newblock On the {S}chr\"{o}dinger maximal function in higher dimension.
\newblock {\em Tr. Mat. Inst. Steklova}, 280:53--66, 2013.

\bibitem{MR3574661}
Jean Bourgain.
\newblock A note on the {S}chr\"{o}dinger maximal function.
\newblock {\em J. Anal. Math.}, 130:393--396, 2016.

\bibitem{MR0576038}
Lennart Carleson.
\newblock Some analytic problems related to statistical mechanics.
\newblock In {\em Euclidean harmonic analysis ({P}roc. {S}em., {U}niv.
  {M}aryland, {C}ollege {P}ark, {M}d., 1979)}, volume 779 of {\em Lecture Notes
  in Math.}, pages 5--45. Springer, Berlin, 1980.

\bibitem{MR3032976}
Thierry Cazenave, Fl\'avio Dickstein, and Fred~B. Weissler.
\newblock Finite-time blowup for a complex {G}inzburg-{L}andau equation.
\newblock {\em SIAM J. Math. Anal.}, 45(1):244--266, 2013.

\bibitem{MR2970037}
Chu-Hee Cho, Sanghyuk Lee, and Ana Vargas.
\newblock Problems on pointwise convergence of solutions to the
  {S}chr\"{o}dinger equation.
\newblock {\em J. Fourier Anal. Appl.}, 18(5):972--994, 2012.

\bibitem{MR0654188}
Bj\"{o}rn E.~J. Dahlberg and Carlos~E. Kenig.
\newblock A note on the almost everywhere behavior of solutions to the
  {S}chr\"{o}dinger equation.
\newblock In {\em Harmonic analysis ({M}inneapolis, {M}inn., 1981)}, volume 908
  of {\em Lecture Notes in Math.}, pages 205--209. Springer, Berlin-New York,
  1982.

\bibitem{MR3702674}
Xiumin Du, Larry Guth, and Xiaochun Li.
\newblock A sharp {S}chr\"{o}dinger maximal estimate in {$\Bbb R^2$}.
\newblock {\em Ann. of Math. (2)}, 186(2):607--640, 2017.

\bibitem{MR3961084}
Xiumin Du and Ruixiang Zhang.
\newblock Sharp {$L^2$} estimates of the {S}chr\"{o}dinger maximal function in
  higher dimensions.
\newblock {\em Ann. of Math. (2)}, 189(3):837--861, 2019.

\bibitem{MR2264734}
Sanghyuk Lee.
\newblock On pointwise convergence of the solutions to {S}chr\"{o}dinger
  equations in {$\Bbb R^2$}.
\newblock {\em Int. Math. Res. Not.}, pages Art. ID 32597, 21, 2006.

\bibitem{MR4899579}
Wenjuan Li, Huiju Wang, and Dunyan Yan.
\newblock Sharp convergence for sequences of {S}chr\"odinger means and related
  generalizations.
\newblock {\em Proc. Roy. Soc. Edinburgh Sect. A}, 155(2):453--469, 2025.

\bibitem{MR3613507}
Renato Luc\`a and Keith~M. Rogers.
\newblock Coherence on fractals versus pointwise convergence for the
  {S}chr\"{o}dinger equation.
\newblock {\em Comm. Math. Phys.}, 351(1):341--359, 2017.

\bibitem{MR3903115}
Renato Luc\`a and Keith~M. Rogers.
\newblock A note on pointwise convergence for the {S}chr\"{o}dinger equation.
\newblock {\em Math. Proc. Cambridge Philos. Soc.}, 166(2):209--218, 2019.

\bibitem{MR2372358}
Changxing Miao, Baoquan Yuan, and Bo~Zhang.
\newblock Well-posedness of the {C}auchy problem for the fractional power
  dissipative equations.
\newblock {\em Nonlinear Anal.}, 68(3):461--484, 2008.

\bibitem{Pan-Sun}
Pan~Yucheng and Sun~Wenchang .
\newblock On the rate of convergence for landau type schrödinger operators.
\newblock 2025+.
\newblock arXiv:2505.24568.

\bibitem{MR4186521}
Lillian~B. Pierce.
\newblock On {B}ourgain's counterexample for the {S}chr\"odinger maximal
  function.
\newblock {\em Q. J. Math.}, 71(4):1309--1344, 2020.

\bibitem{MR0904948}
Per Sj\"{o}lin.
\newblock Regularity of solutions to the {S}chr\"{o}dinger equation.
\newblock {\em Duke Math. J.}, 55(3):699--715, 1987.

\bibitem{MR2549801}
Per Sj\"{o}lin.
\newblock Maximal operators of {S}chr\"{o}dinger type with a complex parameter.
\newblock {\em Math. Scand.}, 105(1):121--133, 2009.

\bibitem{MR2827425}
Per Sj\"olin and Fernando Soria.
\newblock A note on {S}chr\"odinger maximal operators with a complex parameter.
\newblock {\em J. Aust. Math. Soc.}, 88(3):405--412, 2010.

\bibitem{MR0934859}
Luis Vega.
\newblock Schr\"{o}dinger equations: pointwise convergence to the initial data.
\newblock {\em Proc. Amer. Math. Soc.}, 102(4):874--878, 1988.

\bibitem{MR4222395}
Jiye Yuan, Tengfei Zhao, and Jiqiang Zheng.
\newblock On the dimension of divergence sets of {S}chr\"{o}dinger equation
  with complex time.
\newblock {\em Nonlinear Anal.}, 208:Paper No. 112312, 28, 2021.

\end{thebibliography}
\end{document}